\documentclass[12pt]{article}

\usepackage{amsthm,amsmath,amssymb}

\usepackage{graphicx}
\usepackage[T1]{fontenc}

\usepackage[colorlinks=true,citecolor=black,linkcolor=black,urlcolor=blue]{hyperref}


\newcounter{parentnumber}
\theoremstyle{plain}
\newtheorem{theorem}{Theorem}

\theoremstyle{lemma}
\newtheorem{lemma}{Lemma}

\theoremstyle{proposition}

\theoremstyle{corollary}

\theoremstyle{definition}
\newtheorem{definition}{Definition}

\theoremstyle{remark}
\newtheorem{remark}[theorem]{Remark}

 \title{\bf A Coordinate System for Graphs\\
-\\
\small{A New and Efficient Framework for the Graph Isomorphism Problem}  }

\author{Ameneh Farhadian \\
\small Sharif University of Technology\\[-0.8ex]
\small Tehran, I. R. Iran\\}

\date{ \small Mathematics Subject Classifications:  05C62, 05C50, 05C75}

\begin{document}

\maketitle


\begin{abstract}
In this paper, a function  on any pair of graphs  is defined whose properties are  similar to the properties of dot product in  vector space. This  function enables us to define  graph orthogonality  and, also,  a new metric on isomorphism classes of $n$-vertex graphs. Using dot product of  graphs,  a coordinate system  for graphs is provided which  benefits us in graph isomorphism and related problems.

\end{abstract}

\section {Introduction}

Graph isomorphism problem and its derivatives, such as graph matching \cite{conte2004thirty} (in pattern recognition) and graph similarity problem have numerous
applications in many areas such as biology, chemistry \cite{le2004novel}, pattern recognition\cite{conte2004thirty} and web structure mining \cite{chakrabarti2003mining, dehmer2006perspectives}. Graph matching has been the topic of many studies in computer science over the last decades. In graph matching problem, the goal is to find maximum corresponding regions in the given graphs. The graph isomorphism is in fact an exact graph matching. In the graph similarity problem, the main interest is to assign an overall similarity score to indicate the level of similarity between two graphs \cite{zelinka1975certain}.
Maximum common subgraph \cite{bunke1998graph}, edit distance methods \cite{gao2010survey} and measuring distance based on some operations \cite{faudree1994rotation} are some approaches to compare the similarity of two graphs.
 Another approach is using graph kernels \cite{gartner2003graph,neuhaus2007bridging}.\\
Solving the graph isomorphism problem and related problems is  entirely  based on our perception to graph structure. Therefore,  the main step towards an effective solution for these problems is finding a good framework to represent  graphs.

This paper provides a structural  representation for  graphs. The set of graphs is equipped by a dot product which provides us with a graph coordinate representation. Furthermore, the defined dot product donates a new metric to graph space and, also, introduces graph orthogonality. These facilities provides us with a better intuition about graph structure which benefits us in graph isomorphism problem or graph matching in pattern recognition.

Representing  a graph by its adjacency matrix, completely, depends on the vertex ordering. Hence, a graph finds different presentations, due to different reordering of its vertices. This fact
stimulates us to ask whether it is  possible to have a conceptual
description for graphs which is unique for all isomorphic graphs.
Such representation for graphs should be based on the graph structure
instead of defining the edges states which, extremely,
depends on vertices ordering. A complete set of graph
invariants is  a conceptual description for graphs. Also, graph
spectrum is, roughly, what we want, but it does not specify each
graph isomorphism class, uniquely, and the relation of  structure of a graph and its
spectrum is not very clear. 

A real smooth function $ F:(a,b) \longrightarrow R $  can be represented by defining the value of $F$ on any point of the domain. Another possible representation for $F$ is defining  it  in terms of the basis functions, such as $sin(.)$ and $cos(.)$.
However, it is not always easy to compute the Fourier series presentation,
but it  provides  facilities which makes some
complicated problems trivial. Fourier series representation offers
a conceptual and structural view to  functions which is
the base of some technologies such as optics, telecommunications  and
mechanics engineering (vibration). 

The representation of a function in terms of the basis
functions enables us
to be more strong in dealing with functions, either in
 theory or applications.
Can we develop a similar tool for graphs? The first
questions which, naturally, arise are: How can we define a basis for graphs? How can we measure how much a graph is close to a basis element? In this paper, we try to find an answer for these questions and to define a coordinate representation for graphs.
In the second section, a function is defined on any pair of $n$-vertex graph. The  properties of this function on graphs resemble to  properties of the dot product in vector space. Using this function, the  orthogonality is defined for graphs which reveals the different structure of two graphs.  In the third section, a new metric is defined on isomorphism classes of $n$-vertex graphs. In Section 4, a coordinates system is defined for graphs which benefit us in the graph isomorphism and related problems.
\section{ \bf Dot product}
The set of graphs is not a Hilbert space to have a dot product. But, we need something similar to dot product which enables us
 to define a basis for graphs and, also, to measure how much a graph is close to a basis element. The idea of kernel function embeds the set of graphs in a larger Hilbert space which is equipped with a dot product \cite{riesen2010graph, gartner2003graph,neuhaus2007bridging}.
 Here, we want to define something similar to  a dot product on any pairs of $n$-vertex graphs, directly.



We wish this dot product measures the
structural resemblance of any two graphs.
\begin{remark}
In this paper, the matrix representation of a graph $G$ is an $n \times n$ matrix $ A_{G} $ in which there exists $+1$ in $(i,j)$ entry when $v_{i}$ and $v_{j}$ are adjacent and $-1$ otherwise ($ i \neq j $). Diagonal elements are zero.
\end{remark}
Please note that according to this matrix representation for graphs, we have $A_{\overline{G}}=-A_G$.
The trace and transpose of a matrix $A$ is denoted by $tr(A)$ and $A^T$, respectively.\\
Before defining dot product of graphs, the scaler product for graphs is defined.
\begin{definition}
Let $G$ be a simple graph and $r \in R$.\\ $rG$ is a weighted graph, where the weight $e_{u,v}$ is $+r$ for adjacent $\lbrace u,v\rbrace$ in $G$ and is $-r$, otherwise.
 The matrix representation of $rG$  is $rA_G$ where $A_G$ is the matrix representation of the graph $G$.
\end{definition}

\begin{definition}
Let $ A_{G} $, $ A_{H} $ be the matrix representations of two $n$-vertex graphs $ G $ and $H$, respectively. We define
 $$ G . H:= max_P( tr(A_{G}PA_{H}P^{T}))$$
Where $P$ is a permutation matrix. Let $Phase(G,H)$ be the number of permutation matrices $P$ such  that $G.H= tr(A_{G}PA_{H}P^{T})$.
\end{definition}
Clearly, we have  $1 \leq Phase(G,H) \leq n!$. \\
The above dot product is defined in a natural way. A graph is
permuted to other graph as long as the best placement,
maximum number of edge on edge and not edge on not edge assignment, is
found. Two graphs match exactly, if they are isomorphic. We emphasize that this dot product is not exactly a real dot product. But, it is a function on graphs similar to  a dot product with desired properties. For instance, it provides a metric on graphs.  The following properties are resulted directly from the definition.
%

\begin{lemma}\label{first}
Let $G$ and $H$ be two arbitrary graphs on $n$ vertices  and $\Vert G \Vert^2= G.G$.  We have,

\item [a)] $G.H= H.G$.

\item [b)] $ rG.H=G.rH=r(G.H)$ for any $ r \in R$.
\item [c)]$G.H=\overline{G}.\overline{H}$.
\item [d)]$ G.\overline{H}=\overline{G}. H$.
\item [e)] $\frac {G} {\Vert G \Vert}. \frac{H} {\Vert H\Vert}=1$ if and only if $ G \cong H$ .

\end{lemma}
\begin{proof}
Let $A_G$ and $A_H$ be the representation matrices of graphs $G$ and $H$, respectively.

\item [a)] Since $tr(A)=tr(A^T)$ and $tr(AB)=tr(BA)$, we have\\
 $tr(A_GPA_HP^T))=  tr((A_GPA_HP^T)^T) =  tr(PA_HP^TA_G)= tr(A_HP^TA_GP)$
\item [b)] $  tr((rA_G)PA_HP^T)= tr(A_GP(rA_H)P^T))=r. tr(A_GPA_HP^T)$.
\item [c)]$tr(A_GPA_HP^T))=  tr((-A_G)P(-A_H)P^T)$
\item [d)]$tr(A_GP(-A_H)P^T))=  tr((-A_G)PA_HP^T)$
\item [e)]  If $ G \cong H$, there exists a permutation matrix $P$ such that $A_G=PA_HP^T$. Thus,
$\frac {G} {\Vert G \Vert}. \frac{H} {\Vert H\Vert}=$ $\frac {1} {\Vert G\Vert \Vert H \Vert}max_Ptr(A_GPA_HP^T)=$ $\frac{1}{\Vert G\Vert^2} max_Ptr(A_GA_G)=1$.\\
If  $\frac {G} {\Vert G \Vert}. \frac{H} {\Vert H\Vert}=1$, then $max_P(tr(A_GPA_HP^T))=\Vert G \Vert. \Vert H \Vert =tr(A_G^2)$. Thus, $2 max_Ptr(A_GPA_HP^T)= tr(A_G^2)+tr(A_H^2)$. Consequently,
$$max_P(tr(A_G^2+A_H^2-2PA_HP^T))=max_P(tr(A_G-PA_HP^T)^2)=0.$$
Since $A_G-PA_HP^T=(A_G-PA_HP^T)^T$, we have
$$max_P(tr(A_G-PA_HP^T)(A_G-PA_HP^T)^T)=0$$.
We know that  $tr(AA^T)={0}$ if and only if $A$ is a zero matrix. Thus, $A_G-PA_HP^T=0$. It means that graph $G$ is isomorphic to graph $H$.

\end{proof}

\begin{definition}We define  the normalized dot product of two graphs $G$ and $H$  as $G \underline{.}H=\frac {G}{\Vert G \Vert} .\frac{H}{\Vert H \Vert}$
\end{definition}
Clearly, we have $-1 \leq G \underline{.} H \leq 1$ . We saw that $ G \underline{.} H = 1$ if and only if $G$ and $H$ are isomorphic. The normalized dot product of graphs on 4 vertices are shown in Table \ref{tbl_modif}.
\begin{table}[ht]

  \centerline{ \includegraphics[width=10.5cm]{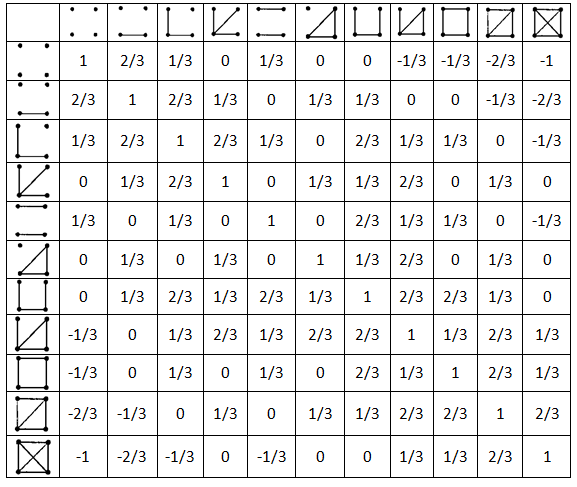}}
  \caption{ \label{tbl_modif} \small The normalized dot product of 4-vertex graphs}
\end{table}


Now, we define graph orthogonality. This new concept offers a new  perception about graph structure. The study of  orthogonal graphs seems to be essential to make our understanding complete about graph structure.

\begin{definition}
Two graphs $ G $ and $ H $ are orthogonal, if $G.H=G.\overline{H}=0$.
\end{definition}
For $n \neq 4k, 4k+1$, it is not possible to have $G.H=0$ and the minimum possible value for $\vert G.H \vert$  is  $ 1$. We call two matrices are quasi-orthogonal, if for two graphs $G$ and $H$ with  $n \neq 4k, 4k+1$, we have   $ G.H=G.\overline{H}=\pm 1$.\\
  If $n= 4k, 4k+1$, the value of zero for dot product of two graphs is possible. Trying to put two orthogonal graphs on another, at  most half of the edge to edge assignments are  successful (edge on edge and not edge on not
edge). 
Two orthogonal graphs are shown in Fig. \ref{ortho}. According to the following lemma, two graphs are
orthogonal if and only if $tr(A_{G}PA_{H}P^{T})$ is invariant to the selection of matrix $P$ and it, constantly, equals to zero.
 Orthogonality of two graphs reveals  perfect different structure of them.
\begin{figure}[ht]
\centerline{\includegraphics[width=5cm]{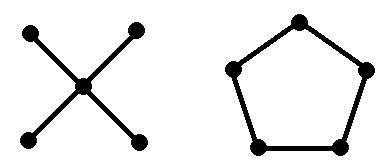}}
\caption{\label{ortho}\small A pair of orthogonal graphs}
\end{figure}
\begin{lemma}  If $G$ and $H$ are orthogonal, then  $Phase(G,H)=n!$.
\end{lemma}
\textit{Proof}: According to definition, if $G.H=G.\overline{H}=0$, then $max_P(A_GPA_HP^T)=max_P(-A_GPA_HP^T)=0 $. Thus,
 for any permutation matrix $P$,  we have
$$A_GPA_HP^T \leq 0 \text{ and} -A_GPA_HP^T \leq 0$$
It follows that $A_GPA_HP^T=0 $ for any permutation matrix $P$.  $\square$\\

\section{A metric space for graphs}
Some different distances are defined on the set of  isomorphism classes of graphs which only some of them are metrics. For instances, the distances  defined in \cite{zelinka1975certain, bunke1998graph} are metrics which are based on maximum common subgraph. In \cite{lovasz2012large}, a metric is defined for the cut of graphs.
Here, a new metric  on the set of $n$-vertex graphs  is introduced..

\begin{definition} We define for any two $n$-vertex graphs $G$ and $H$.

$$d(G,H):= \Vert G\Vert^{2}+\Vert H \Vert^{2}-2G.H$$ where $\Vert  G \Vert  =\sqrt {G{.}G}$.

\end{definition}
\begin{theorem}$d$ is a metric on the set of isomorphism classes of $n$-vertex graphs.
 \end{theorem}

\begin{proof} We should check the following properties hold true for any graphs $G, G_{1}, G_{2}$ and $H$
\begin{enumerate}

\item $ d(G, H)\geq 0 $
\item $d(G,H)=0$ if and only if $G\cong H$
\item $ d(G_{1},G_{2}) \leq d(G_{1},H)+d(G_{2},H)$
\end{enumerate}
Please note that

$$d(G,H) =\Vert G\Vert^2+\Vert H \Vert^2-2G.H=\Vert G\Vert^2+\Vert H \Vert^2-2max_P(tr( A_GPA_HP^T))$$
$$=min_P\lbrace tr(A_G^2+A_H^2-2 A_GPA_HP^T) \rbrace=$$
$$=min_P \lbrace tr(A_G^2+A_H^2-A_GPA_HP^T-PA_HP^TA_G) \rbrace$$
Thus, $$d(G,H)=min_P \lbrace tr(A_G-PA_HP^T)^2 \rbrace$$

(1) We know that  $tr(AA^T) \geq 0$ for any real matrix $A$.  Since $d(G,H)=min_P \lbrace tr((A_G-PA_HP^T)(A_G-PA_HP^T)^T) \rbrace$,  we have $d(G,H) \geq 0 $.

(2) We know that $tr(AA^T)=0$, if and only if $A$ is a zero matrix. Thus, $d(G,H)=0$, if and only if  $A_G-PA_HP^T$   is a zero matrix. It means that  there exists a permutation matrix $P$ such that $A_G=PA_HP^T$, i.e. $G$ is isomorphic to $H$.

(3) Let $P$ and $Q$ be the permutation matrices such that
  $H.G_1=tr(A_{H}PA_{G_1}P^T) $  and  $   H.G_2=tr(A_{H}QA_{G_2}Q^T)$.
The possible values for the entries of matrix $C=A_{H}-PA_{G_1}P^T$ are $2$, $-2$ and $0$. We have $c_{i,j}=0$, if $
(A_{H})_{i,j}=(PA_{G_1}P^T)_{i,j}$. If $
(A_{H})_{i,j} \neq (PA_{G_1}P^T)_{i,j}$, then  $c_{i,j}=-2$ or $2$. Thus, $$\frac{1}{4}\sum_{ij} c_{ij}^2=\frac{1}{4}tr(CC^T)=\frac{1}{4}tr((A_{H}-PA_{G_1}P^T)^2)$$ indicates the number of non-zeros entries of $C=A_H-PA_{G_1}P^T$, i.e. the number of entries which $A_H$ is different from $PA_{G_1}P^T$.  Also, $\frac{1}{4}tr((A_{H}-QA_{G_2}Q^T)^2)$ indicates the number of entries which  $A_{H}$  is different from $QA_{G_2}Q^T$.

Clearly, if  $(A_H)_{ij}=(PA_{G_1}P^T)_{ij}$ and  $(A_H)_{ij}=(QA_{G_2}Q^T)_{ij}$, then
 $(PA_{G_1}P^T)_{ij}=(QA_{G_2}Q^T)_{ij}$. In opposite, we have $(PA_{G_1}P^T)_{ij} \neq (Q A_{G_2}Q^T)_{ij}$, if  $(PA_{G_1}P^T)_{ij} \neq (A_H)_{ij}$ or   $(QA_{G_1}Q^T)_{ij} \neq (A_H)_{ij}$. Thus,  the number of  entries of $PA_{G_1}P^T$  which are different from $QA_{G_2}Q^T$ is at most $\frac{1}{4}tr((A_{H}-PA_{G_1}P^T)^2)+\frac{1}{4}tr((A_{H}-QA_{G_2}Q^T)^2)$. In other words,
$$\frac{1}{4}tr((PA_{G_1}P^T-QA_{G_2}Q^T)^2)<\frac{1}{4}tr((A_{H}-PA_{G_1}P^T)^2)+\frac{1}{4} tr((A_{H}-QA_{G_2}Q^T)^2)$$
$$tr(PA_{G_1}P^T-QA_{G_2}Q^T)^2\leq tr(A_{H}-PA_{G_1}P^T)^2+tr(A_{H}-QA_{G_2}Q^T)^2$$
By  manipulating,
$$tr((A_{G_1}-P^TQA_{G_2}Q^TP)^2)\leq tr((A_{H}-PA_{G_1}P^T)^2)+tr((A_{H}-QA_{G_2}Q^T)^2)$$
Clearly, $P^TQ$ is a permutation matrix. Substituting  $P^TQ$ by $S$, we have
$$tr((A_{G_1}-SA_{G_2}S^T)^2)\leq tr((A_{H}-PA_{G_1}P^T)^2)+tr((A_{H}-QA_{G_2}Q^T)^2)$$
According to definition, we have $d(G_1,G_2)=min_P(tr(A_{G_1}-PA_{G_2}P^T)^2)$. Thus, we have $d(G_1,G_2) \leq tr((A_{G_1}-SA_{G_2}S^T)^2 )$ for any permutation matrix $S$. Therefore,

$$d(G_1,G_2) \leq tr((A_{G_1}-SA_{G_2}S^T)^2) \leq d(G_{1},H)+d(G_{2},H) $$\end{proof}

 The study of  topology that  metric $d$  induces on the set of graphs is suggested for the future work.


\subsection{Dot product of graphs with different order}
In Section 2, the dot product of graphs with the same order was defined.
Here, the graph dot product  is extended for graphs with different order. In \cite{ kocay1982some, erdHos1979strong}, the number of subgraph of a graph is,  merely, counted as subgraph algebra. Here, there is a more general approach.
\begin{definition}
Let $ A_{G} $, $ A_{H} $ be, respectively, the matrix representations of two graphs $ G $ and $H$ with respectively $n$ and $k$ vertices such that  $n>k$.  First,  we add extra zero rows and columns  to matrix $A_H$ to have two matrices with the same size, then we define
 $$ G.H:= max_P( tr(A_{G}PA_{H}P^{T}))$$
where $P$ is a permutation matrix.\\
 Assuming that $f$ is a mapping from $V(H)$ to $V(G)$ and  $H_f$ is the induced subgraph on $f(V(H))$,  $Phase(G,H)$  denotes  the number of mapping $f :V(H) \to V(G)$ such that $G.H_f=G.H$.

\end{definition}

\begin{lemma}\label{second}
Let $G$ and $H$ be arbitrary graphs.

\item [a)] ${G} .{H} =H.H $ if and only if $H$ is a subgraph of $G$.
\item [b)]  If $H$ is a subgraph of $G$, then $Phase(G,H)$/Aut($H$) is the number of occurrences of the subgraph $H$ in $G$.

\end{lemma}
\begin{proof}
\item [a)] It can be easily checked that  $max_P( tr(A_{G}PA_{H}P^{T}))= max_{H_f}( H.H_f )$ where $ H_f$ is a $k$-vertex subgraph of graph $G$. Thus, the maximum value of $G.H$ occurs if and only if there exists a subgraph $H_f$  in $G$ isomorphic to $H$.
In opposite, if there is a subgraph $H_f$ isomorphic to $H$ in $G$, then $G.H=H.H_f=H.H$.
\item [b)] If $H$ is a subgraph of graph $G$, we have $Phase(G,H)= $
 $$ \vert \lbrace H_f \vert \text{$H_f$ is a $k$-vertex subgraph of $G$ isomorphic to $H$} \rbrace  \vert . \vert Aut(H) \vert$$
\end{proof}
\section{Coordinate Representation }
We saw how graph dot product  reveals significant information about the structure of a pair of graphs, such as being isomorphic, close to  isomorphic or completely different structure by orthogonality. Now, we want to identify a graph according to its dot product by a a set of graphs.
\begin{definition}
Let $S=(H_{1},..,H_{t})$ be an ordered set of graphs.
The \textit{coordinates of a graph $G$} with respect to the set $S$ is
$(G.H_{1},\cdots,G.H_{t})$.
\end{definition}
Please note that for any graph $H_i$, the $Phase (H_i,G)$ is also computed.
\begin{definition}
Let $\Gamma$ be a set of graphs.
A set $S$ of graphs is a \textit{basis} for $\Gamma$, if
 any graph in $\Gamma$ has a unique  coordinates with respect to $S$.
 \end{definition}
\begin{figure}[ht]
\centerline{\includegraphics[width=9cm]{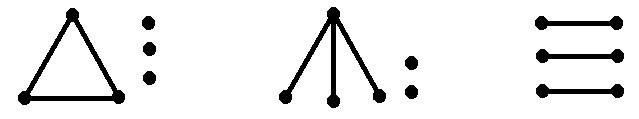}}
\caption{\label{f}\small A basis for 6-vertex graphs}
\end{figure}
A set of graphs is a basis for itself. A basis for 6-vertex graphs is shown in Fig. \ref{f}. The coordinate representation of graphs is a useful tool to deal with graphs. Two isomorphic graphs share the same coordinates. Thus, to test the isomorphism of two graphs, it is sufficient  to check their coordinates. Also, the coordinates of graphs indicate how similar two graphs are.

In the conventional  approach for checking the isomorphism of two graphs, one tries to find a one to one correspondence between the vertices of two graphs. Against, in a coordinate system, it is sufficient to compute the coordinates and compare them. 

Although, computing the dot product of two graphs, in general, is as hard as graph isomorphism problem, but, the basis is a fixed set of graphs. The fixedness of basis  is the outstanding advantage of this new approach. The fixedness of basis elements  makes it possible to have some pre-computations, if it is needed, or  have a dedicated physical  infrastructure which computes the dot product of any graphs with a fixed basis element. The implementation of this infrastructure  is given in Fig. \ref{hardware}.\\
 More importantly, the basis elements can be chosen cleverly to decrease the computational complexity. For example, there are families of graphs that the computation of their dot product  to any arbitrary graph can be done in  polynomial time, such as  bounded order graphs, star graphs $S_k$,  $K_r \cup K_{n-r}$ or $K_r \cup \overline{K}_{n-r}$($r$ is a fixed integer).\\
 To study the computational complexity of the graph isomorphism problem, we should find a suitable basis for graphs. For instance, if we can find a basis for $n$-vertex graphs whose cardinality is a polynomial in terms of $n$ and the  dot product  of any graph with basis elements can be done in polynomial time in terms of $n$, in fact,  we have found a polynomial time algorithm for the graph isomorphism problem.

\begin{figure}[ht]
\centerline{\includegraphics[width=10cm]{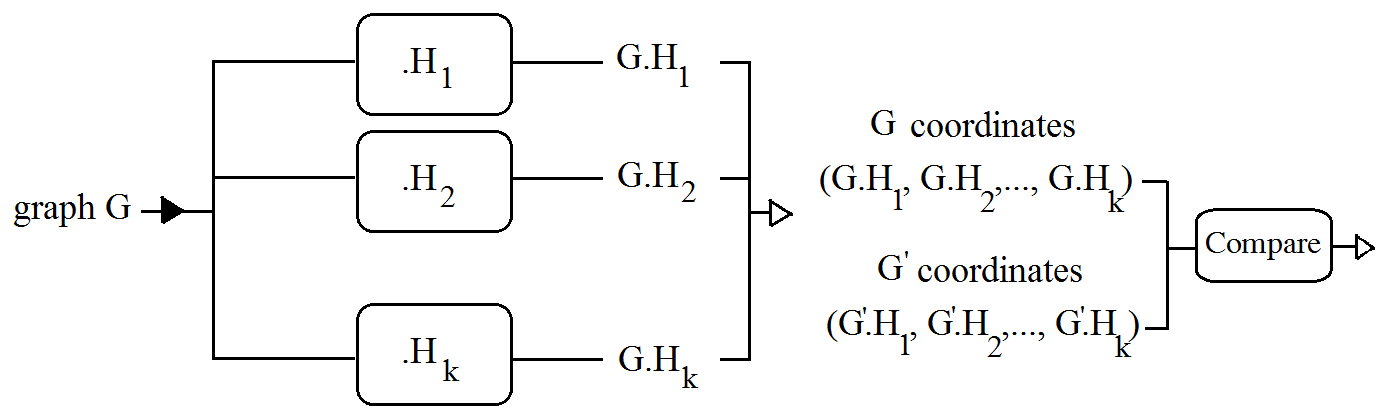}}
\caption{\label{hardware}\small A framework for solving the
graph isomorphism problem or graph matching ($H_{i}$ s  are basis elements )}
\end{figure}

Another application of the defined graph coordinate system is graph clustering, i.e. classifying a set of graphs according to their class of isomorphism. Assume we have a set of $m$ graphs with $n$ vertices. We want to classify them into isomorphism classes.
We need to compare any pair of graphs. Hence, the algorithm of checking isomorphism of two graphs should be called  $O(m^2)$ times. In opposite, using the graphs coordinates, it is sufficient to compute the graph coordinates and compare them. Thus, clustering of graphs which needs $O(m^2)$ comparison of graphs reduces to a sorting problem with $O(\log(m))$ time complexity . \\
The graph coordinate representation is useful not only in graph isomorphism problem, but also in graph matching and graph similarity problems. The closer coordinates, the more similar structure. Therefore, the coordinate system of graphs, also, benefits us in classifying and clustering graphs in inexact cases.

\subsection{ A basis for almost all $n$-vertex graphs }
 In the probability space of graphs on $n$ labeled vertices in which the edges are chosen independently, with probability $p=1/2$, we say that almost every graph $G$ has a property $Q$ if the probability that $G$ has $Q$ tends to 1 as $n \rightarrow \infty$.\\
It has been shown that almost every $n$-vertex graph is, uniquely,  determined by the number of occurrence of its $3 \log_2{n}$-vertex subgraphs \cite{farhadian2018almost}.
 According  to Lemma \ref{second}, $G.H$ indicates whether $H$ is a subgraph of $G$. Additionally, if $H$ is a subgraph of $G$, the number of  copies of  subgraph $H$ occurred in $G$ can be obtained from $Phase(G,H)$.
Therefore, the set of graphs with $3\log_2{n}$ vertices  is a basis for almost all $n$-vertex graphs.
%
%
%
%


\end{document}